\documentclass[11pt]{article}

\usepackage{amsthm,amsfonts,amssymb,amsmath}
\usepackage{enumerate}
\usepackage{fullpage}

\usepackage[colorlinks=true]{hyperref}
\newtheorem{thm}{Theorem}[section]

\newtheorem{theorem}[thm]{Theorem}

\newtheorem{lemma}[thm]{Lemma}

\theoremstyle{definition}

\newtheorem{example}[thm]{Example}

\theoremstyle{remark}

\newcommand{\BC}{{\mathbb {C}}}

\newcommand{\BF}{{\mathbb {F}}}

\newcommand{\BR}{{\mathbb {R}}}

\newcommand{\CC}{\mathbb{C}}
\newcommand{\RR}{\mathbb{R}}
\newcommand{\ZZ}{\mathbb{Z}}
\newcommand{\QQ}{\mathbb{Q}}

\newcommand{\PP}{\mathbb{P}}

\newcommand{\CA}{{\mathcal {A}}}

 \newcommand{\CCC}{{\mathcal {C}}}
\newcommand{\CD}{{\mathcal {D}}}

\newcommand{\CH}{{\mathcal {H}}}

\newcommand{\CL}{{\mathcal {L}}}
\newcommand{\CM}{{\mathcal {M}}}

\newcommand{\CO}{{\mathcal {O}}}

\newcommand{\CX}{{\mathcal {X}}}
\newcommand{\CY}{{\mathcal {Y}}}

\newcommand{\an}{{\mathrm{an}}}

\newcommand{\Aut}{{\mathrm{Aut}}}

\renewcommand{\div}{{\mathrm{div}}}
\newcommand{\wdiv}{{\mathrm{wdiv}}}

\newcommand{\rank}{{\mathrm{rank}}}

\newcommand{\Pic}{\mathrm{Pic}}

\newcommand{\tor}{\textnormal{tor}}

\newcommand{\vol}{{\mathrm{vol}}}

\newcommand{\wh}{\widehat}

\newcommand{\pair}[1]{\langle {#1} \rangle}

\newcommand{\ds}{\displaystyle}

\newcommand{\ol }{\overline}

\newcommand{\lra}{\longrightarrow}

\newcommand{\CLL}{{\overline{\mathcal L}}}

\newcommand{\OL}{{\overline{L}}}
\newcommand{\OM}{{\overline{M}}}
\newcommand{\OD}{{\overline{D}}}

\newcommand{\OK}{{\overline{K}}}

\newcommand{\lb}{\mathcal{L}}              
\newcommand{\mb}{\mathcal{M}}              
\newcommand{\lbb}{\overline{\mathcal{L}}}  

\newcommand{\hlb}{hermitian line bundle}

\newcommand{\Fal}{\textnormal{Fal}}

\newcommand{\VCV}{\Vert\cdot\Vert}

\DeclareMathOperator{\Adeg}{\widehat{\textnormal{deg}}}

\DeclareMathOperator{\Spec}{\textnormal{Spec}}

\begin{document}

\title{Quantitativity in the Mordell Conjecture}
\author{Xinyi Yuan}

\maketitle






\begin{abstract} 
The Mordell conjecture asserts that there are only finitely many rational points on a smooth projective curve of genus at least two over a number field. 
The uniform Mordell problem asks for suitable upper bounds on the number of rational points in the Mordell conjecture. In this survey, we will introduce various recent developments on this uniformity problem.  
\end{abstract}

\tableofcontents

\section{Introduction}

Classical Diophantine problems aim to find rational solutions of polynomial equations with rational coefficients. In modern terminology, the goal of Diophantine geometry is to study rational points of algebraic varieties over number fields. 
The celebrated Mordell conjecture, proved by Faltings in 1983, is the following statement. 

\begin{theorem}[Faltings]
Let $C$ be a  curve of genus at least two over a number field $K$. 
Then $C(K)$ is finite. 
\end{theorem}

As a convention in this paper, a \emph{curve} always means a smooth and geometrically connected projective variety of dimension 1 over its base field. 

The conjecture was raised by Mordell \cite{Mor22} when he proved the finite generation of the group of the rational points of elliptic curves over $\QQ$ in 1922. Faltings' proof of the conjecture in \cite{Fal83} signified a milestone in the history of Diophantine geometry. 
By an analogue to the Thue--Siegel--Roth theorem in Diophantine approximation, Vojta \cite{Voj91} gave a new proof of the Mordell conjecture in 1991. 
Inspired by Faltings' proof, Lawrence--Venkatesh \cite{LV20}
gave a third proof of the Mordell conjecture in terms of $p$-adic Hodge theory.
We will come back to Vojta's proof later, since it will be used in our main results. 

Once we know that the set $C(K)$ is finite, a natural question is to give a suitable upper bound on the order of this set in terms of natural geometric and arithmetic invariants of $C$ over $K$. 
In this direction, we have the following uniform version of the Mordell conjecture, which solves a problem proposed by Mazur \cite[p. 234]{Maz86}.

\begin{theorem}[Vojta \cite{Voj91}, Dimitrov--Gao--Habegger \cite{DGH21}, K\"uhne \cite{Kuh}]
\label{uniform Mordell}
Let $g>1$ be an integer. 
Then there are positive constants $c_1(g)$ and $c_2(g)$ depending only on $g$ such that for any curve $C$ of genus $g$ over a number field
$K$, we have
$$
 |C(K)| \leq c_1(g) c_2(g)^{r}.
$$ 
Here $r=\rank\, J(K)$ is the Mordell--Weil rank of the Jacobian variety $J$ of $C$ over $K$.
\end{theorem}

The theorem is proved by a combination of the following bounds. 
\begin{enumerate}
\item (Large points)
Vojta's proof of the Mordell conjecture actually gives a deep inequality concerning distribution of rational points of large heights, which particularly implies an upper bound on the number of rational points of large heights. Note that this is sufficient for the Mordell conjecture by the Northcott property of heights.
The upper bound was refined by de Diego \cite{dD97} and R\'emond \cite{Rem00a}.
\item (Small points)
The works \cite{DGH21, Kuh} prove a uniform Bogomolov conjecture, which gives an upper bound on the number of rational points of small heights. 
By a simple argument of sphere packing, the uniform Bogomolov conjecture also bounds the number of all rational points complementary to (1). 
\end{enumerate}

There are two other approaches of the uniform Bogomolov conjecture.
In \cite{Yua21}, the author gave a different proof of the uniform Bogomolov conjecture in terms of the theory of adelic line bundles on quasi-projective varieties of Yuan--Zhang \cite{YZ}, which works over global fields;
In \cite{LSW}, Looper--Silverman--Wilms  proved the uniform Bogomolov conjecture over function fields with explicit constants. 
The work \cite{LSW} does not prove the case of number fields due to many obstacles  at archimedean places.  

We will explain all the above results in more details. For the moment, let us remark that Vojta's proof of the Mordell conjecture was extended by Faltings \cite{Fal91, Fal94} to prove the Bombieri--Lang conjecture for subvarieties of abelian varieties.
In this direction, a generalization of Theorem \ref{uniform Mordell} to algebraic families of subvarieties of abelian varieties was proved by  
Gao--Ge--K\"uhne \cite{GGK21}.

Let us return to  Theorem \ref{uniform Mordell}. 
A natural problem is to express the constants $c_1(g)$ and $c_2(g)$ explicitly in terms of $g$. Over function fields of characteristic 0, this problem was solved by Jiawei Yu in  \cite{Yu23} with $c_1(g)=16g^2 + 32g + 184$ and $c_2(g)=20g$ by applying the quantitative Bogomolov conjecture of \cite{LSW}. 
Over number fields, 
we have the following theorem by Jiawei Yu, Shengxuan Zhou, and the author. 

\begin{theorem}[Yu--Yuan--Zhou \cite{YYZ}]
\label{quant Mordell}
For any curve $C$ of genus $g>1$ over a number field $K$, we have
$$
 |C(K)| \leq 10^{13} g^8 \big(1+\frac{3\log g}{g}\big)^{r}.
$$ 
Here $r=\rank\, J(K)$ is the Mordell--Weil rank of the Jacobian variety $J$ of $C$ over $K$.
\end{theorem}

As in the Mordell--Lang philosophy, Theorem \ref{quant Mordell} is also valid for subgroups of $J(\ol K)$ of finite rank. 
In fact, \cite{YYZ} actually proves the following result. 

\begin{theorem}[Yu--Yuan--Zhou \cite{YYZ}]
\label{quant Mordell 2}
Let $C$ be a curve of genus $g>1$ over $\CC$, $J$ be the Jacobian variety of $C$ over $\CC$, $\Lambda$ be a subgroup of $J(\CC)$, 
$\beta$ be a divisor on $C$ of degree 1, and 
$i_\beta:C\to J$ be the Abel--Jacobi map $x\mapsto x-\beta$.
Then 
$$
 |i_\beta(C)(\CC)\cap \Lambda | \leq  10^{13} g^8 \big(1+\frac{3\log g}{g}\big)^{\dim_{\QQ}(\Lambda\otimes_\ZZ\QQ)}.
$$ 
\end{theorem}

The theorem has a very non-arithmetic statement. However,  by a simple specialization argument, it is immediately reduced to the case where $(C,J, \beta, \Lambda)$ is defined over a number field $K$. Then it has the same flavor as  Theorem \ref{quant Mordell}.

In the theorem, the Manin--Mumford constant $c_1(g)=10^{13} g^8$, which bounds the number of torsion points on $C$ by taking $\Lambda=J(\CC)_{\tor}$, 
 is a polynomial in $g$. We suspect that it is far from being optimal. 
The Vojta constant $\displaystyle c_2(g)=1+\frac{3\log g}{g}$, which essentially comes from Vojta's proof of the Mordell conjecture, tends to 1 as $g\to\infty$. This responds to a question of Gao and Habegger (cf. \cite[Que. 1.19]{Hab22}).  

The proof of Theorem \ref{quant Mordell} (and Theorem \ref{quant Mordell 2}) still consists of bounding points by different heights.
For large points, we prove a quantitative version of Vojta's inequality by carefully analyzing his original proof. 
For small points, it is difficult to obtain explicit constants from the proofs of \cite{DGH21, Kuh} and \cite{Yua21}, so our approach is to carry out the idea of Looper--Silverman--Wilms \cite{LSW} to prove a quantitative Bogomolov conjecture over number fields. 
To do this, we have overcome all the archimedean obstacles by proving many desired inequalities on invariants of compact Riemann surfaces.  
To obtain a sharp Vojta constant $c_2(g)$, we improve the bounds on points of medium heights by a refined version of the quantitative Bogomolov conjecture inspired by \cite{Yua21}. 

The goal of this survey is to explain some background and  ideas of the proof of Theorem \ref{quant Mordell}. Our topic is closely related to the previous surveys of Gao \cite{Gao21} and Habegger \cite{Hab22}. 
While  \cite{Gao21} mainly concerns the proof of the uniformity result in Theorem \ref{uniform Mordell}, and \cite{Hab22} gives a panoramic view of various related problems and results on the Mordell conjecture, 
our current survey mainly concerns the proof of the quantitativity result in Theorem \ref{quant Mordell}.
We refer the readers to these two surveys for more aspects of the topic.

\section{Preliminary on heights and arithmetic intersections}
The goal of this section is to review some basic theory of heights and arithmetic intersection theory. 

Throughout this paper, by a \emph{variety} over a field $F$, we mean an integral scheme separated of finite type over $F$. By a \emph{curve} over a field $F$, we mean a smooth and geometrically connected projective variety of dimension 1 over $F$. In Diophantine problems, we always assume that the genus of the curve is $g>1$.

We usually write tensor products of line bundles additively. For example, $aL-bM$ means $L^{\otimes a}\otimes M^{\otimes (-b)}$ for line bundles $L,M$ and integers $a,b$. We allows $a,b$ to be rational numbers to get $\QQ$-line bundles.

\subsection{Heights}
We start with the classical theory of Weil heights and N\'eron--Tate heights. 

\subsubsection{Weil heights}
Let us briefly recall the definition of Weil heights on projective varieties and canonical heights on abelian varieties. For a systematic introduction, we refer to \cite{Ser89, HS00, BG}. 

Let $K$ be a number field. 
Denote by $M_K$ the set of places of $K$. For any place $v$ of $K$, normalize the absolute value $|\cdot|_v$ on the completion $K_v$ as follows.
\begin{enumerate}
\item If $v$ is an archimedean place, take $|\cdot|_v$ to be the usual absolute value on $K_v\cong\RR,\CC$;
\item If $v$ is a non-archimedean place, set $|a|_v=\big(\#(O_{K_v}/aO_{K_v})\big)^{-1}$ for any $a\in O_{K_v}$. 
\end{enumerate}
The valuations satisfy the product formula 
$$
\prod_{v\in M_K} |a|_v^{\epsilon_v} =1, \quad \forall \ a\in K^\times. 
$$
Here $\epsilon_v=1$ for all real or non-archimedean places $v$; $\epsilon_v=2$ for all complex places $v$.

Let $\PP^n$ be the projective space over $K$.
The \emph{standard height function $h: \PP^n(\overline K)\to \RR$} is defined to be
$$h(x_0, x_1, \cdots, x_n)=\frac{1}{[K':\QQ]}\sum_{w\in M_{K'}} \epsilon_w \log \max\{|x_0|_w, |x_1|_w,\cdots, |x_m|_w \},$$
where $K'$ is a finite extension of $K$ containing all the coordinates $x_i$, and the sum is over all places $w$ of $K'$. 

Let $X$ be a projective variety over $K$, and $L$ be any ample line bundle on $X$. 
Let $i: X\to \PP^n$ be any morphism such that $i^*\CO(1)=L^{\otimes d}$ for some $d\geq 1$. 
We obtain a height function $h_{L,i}= \frac1d h\circ i : X(\overline K)\to \RR$ as the composition of 
$i: X(\overline K)\to \PP^n(\overline K)$ and $\frac1d h: \PP^n(\overline K)\to \RR$. 
It depends on the choices of $d$ and $i$.

More generally, let $L$ be any line bundle on $X$. We can always write $L= A_1\otimes A_2^{\otimes (-1)}$ for two ample line bundles $A_1$ and $A_2$ on $X$. For $k=1,2$, let $i_k: X\to \PP^{n_k}$ be any two morphisms such that $i_k^*\CO(1)=A_k^{\otimes d_k}$ for some $d_k\geq 1$. 
We obtain a height function $h_{L,i_1,i_2}= h_{A_1,i_1}-h_{A_2,i_2}: X(\overline K)\to \RR$. 
It depends on the choices of $(A_1, A_2, i_1,i_2)$.
However, the following result asserts that it is unique up to bounded functions.

\begin{theorem}[Weil's height machine] \label{Weil}
The above construction $L\mapsto h_{L,i_1,i_2}$ gives a group homomorphism
$$
\CH: \mathrm{Pic}(X) \longrightarrow 
\frac{\{ \mathrm{functions\ }\phi: X(\overline K)\to \RR \}}
{ \{ \mathrm{bounded\ functions\ }\phi: X(\overline K)\to \RR \}}.
$$
\end{theorem}
A function $h_L: X(\overline K)\to \RR $ in the class $\CH(L)$ in the theorem is called a \emph{Weil height function} corresponding to $L$.

\subsubsection{N\'eron--Tate heights}  \label{canonical height}
In the following, we refer to  \cite{Mum08, Mil08} for the basics of abelian varieties, to 
\cite{Mil86} for the basics of Jacobian varieties, and  to 
 \cite{Ser89, HS00, BG} for the basics of N\'eron--Tate heights. 

Let $A$ be an abelian variety over a number field $K$. Denote by $[m]:A\to A$ the multiplication by an integer $m$. 
Let $L$ be a symmetric and ample line bundle on $A$. 
Here $L$ is called \emph{symmetric} if $[-1]^*L\cong L$, which implies $[m]^*L\cong L^{\otimes m^2}$ for all integers $m$.  
Let $h_L: A(\overline K)\to \RR $ be any Weil height function corresponding to $L$.
The \emph{N\'eron--Tate height} (or the \emph{canonical height})
$$\hat h_{L}: A(\overline K)\lra \RR$$
with respect to $L$ is defined by Tate's limit
$$\hat h_{L}(x) =\lim_{n\rightarrow \infty} \frac{1}{4^n}h_{L}(2^n x), \quad 
x\in A(\overline K).$$
An easy argument proves the convergence and independence of the choice of the Weil height function $h_L$. Moreover, $\hat h_L:A(\overline K)\to \RR$ is also a Weil height function associated to $L$. 
We further have the following quadraticity and positivity properties.

\begin{theorem}\label{canonical}
\begin{enumerate}
	\item $\hat h_L(x)\geq 0$ for any $x\in A(\overline K)$, and the equality holds if and only if $x$ is torsion.
	\item 
The function $\hat{h}_L: A(\overline{K})\to\RR$ is  {quadratic} in the sense that
 it satisfies the parallelogram rule 
 $$
 \hat{h}_L(x+y)+\hat{h}_L(x-y)=2\left(\hat{h}_L(x)+\hat{h}_L(y)\right),\quad  \forall x,y\in A(\overline{K}).
 $$
 Moreover, we have
 $$
 \hat{h}_L(mx)=m^2\hat{h}_L(x),\quad  \forall m\in\ZZ,\quad  \forall x\in A(\overline{K}).
 $$
 \item 
 The quadratic form $\hat{h}_L: A(\overline{K})\otimes_\ZZ\RR \to\RR$, induced by the height function $\hat{h}_L: A(\overline{K})\to\RR$ via the quadraticity, is positive definite on the real vector space $A(\overline{K})\otimes_\ZZ\RR$. 
\end{enumerate}
\end{theorem}

Let $K$ be a number field. Let $C$ be a curve of genus $g>1$ over $K$. 
Denote by $J$ the Jacobian variety of $C$ over $K$.  
As the multiplication
 $[2g-2]\colon J({\ol K})\to  J({\ol K})$
is surjective, there is a line bundle $\alpha$ on $C_{\ol K}$ such that $(2g-2)\alpha$ is isomorphic to the canonical sheaf $\omega_C$. 
Replacing $K$ by a finite extension if necessary, we can assume that $\alpha$ is actually a line bundle on $C$. 
We will see that this process does not change our definition of heights. 

Consider the Abel--Jacobi embedding
$$
i_\alpha: C\lra J, \quad  x\longmapsto x-\alpha.
$$
Recall that the theta divisor on $J$ is given by
$$
\theta_\alpha= \underbrace{i_\alpha(C)+\dots + i_\alpha(C)}_{g-1 \text{ copies}}.$$
It is well-known that $\theta_\alpha$ is ample and gives a principal polarization of $J$. 
By \cite[p. 74, eq. (1)]{Ser89}, $\theta_\alpha$ is symmetric in the sense that $[-1]^*\theta_\alpha$ is linearly equivalent to $\theta_\alpha$. 

Denote $J(\ol K)_\RR=J(\ol K)\otimes_\ZZ \RR$. 
As in Theorem \ref{canonical}, the N\'eron--Tate height 
$$\hat h= \hat h_{\theta_\alpha}\colon J(\ol K)\lra \RR$$
is quadratic and induces a positive definite quadratic form
$$\hat h= \hat h_{\theta_\alpha}\colon J(\ol K)_\RR\lra \RR.$$
For convenience, we take the normalization 
$$
\hat h(x)_K=[K:\QQ]\cdot \hat h(x), \quad x\in J(\OK). 
$$
The factor $[K:\QQ]$ will bring convenience when comparing heights with arithmetic intersection numbers.

Denote the N\'eron--Tate height pairing
$$\pair{x,y}\colon J(\overline{K})\times J(\overline{K})\longrightarrow \BR$$
over $K$
by 
$$ \pair{x,y}= \frac12 \big(\hat{h}(x+y)_K-\hat{h}(x)_K-\hat{h}(y)_K\big), \quad x, y\in J(\ol K).$$
The associated norm of the pairing is given by  
$$|x| := \sqrt{\hat{h}(x)_K}, \quad x\in J(\ol K).$$
These definitions extend to $x, y\in J(\ol K)_\RR$ canonically. 

We apply $\hat h(\cdot)$, $\hat h(\cdot)_K$, $|\cdot|$ and $\pair{\cdot,\cdot}$ to $C(\ol K)$ via the embedding $i_\alpha: C\to J$. 
Here $\alpha$ is a line bundle on $C_{\ol K}$ with 
$(2g-2)\alpha\cong \omega_{C/K}$.
For example, for $x\in C(\OK)$, we have
$$\hat h(x)=\hat h(i_\alpha(x))=\hat h(x-\alpha)=\hat h_{\theta_\alpha}(x-\alpha), \quad
|x|^2=\hat h(x)_K=[K:\QQ]\hat h(x).
$$

\subsection{Arithmetic intersection of hermitian line bundles}
Let us introduce some terminology of Arakelov geometry developed by Arakelov \cite{Ara74}, Deligne \cite{Del}, and Gillet--Soul\'e \cite{GS90}. 

Let $\CX$ be an arithmetic variety of dimension $n+1$, i.e. a projective and flat integral scheme over $\Spec \ZZ$ of absolute dimension $n+1$. 
A \emph{hermitian line bundle} on $\CX$ is a pair  $\CLL=(\CL,\VCV)$, where $\CL$ is a line bundle on $\CX$, and $\VCV$ is a smooth hermitian metric of $\CL(\CC)$ on the complex variety $\CX(\CC)$ invariant under the complex conjugation. 
If $\CX(\CC)$ is singular, the smoothness of the metric means that the metric is locally equal to the pull-back of a smooth metric via a closed embedding into a smooth complex manifold. 

Let $\lbb_1, \lbb_2, \cdots, \lbb_{n+1}$ be $n+1$ hermitian line bundles on $\CX$, and let $s_{n+1}$ be any non-zero rational section of $\lb_{n+1}$ on $\CX$. 
The intersection number is defined inductively by 
\begin{eqnarray*}
 \lbb_1\cdot \lbb_2 \cdots \lbb_{n+1}
= \lbb_1\cdot \lbb_2 \cdots \lbb_{n} \cdot \wdiv(s_{n+1})
- \int_{\CX(\CC)} \log \|s_{n+1}\| c_1(\lbb_1)\cdots c_1(\lbb_n).
\end{eqnarray*}
The right-hand side depends on the Weil divisor $\wdiv(s_{n+1})$ linearly, so it suffices to explain the case that $\CD=\wdiv(s_{n+1})$ is irreducible (and reduced). 

If $\CD$ is horizontal in the sense that it is flat over $\ZZ$, then
$$
\lbb_1\cdot \lbb_2 \cdots \lbb_{n} \cdot \CD= \lbb_1|_\CD\cdot \lbb_2|_\CD \cdots \lbb_{n}|_\CD
$$
is an arithmetic intersection on $\CD$. 
If $\CD$ is a vertical divisor in the sense that it is a variety over $\BF_p$ for some prime $p$, then the integration is zero, and 
$$
\lbb_1\cdot \lbb_2 \cdots \lbb_{n} \cdot \CD= (\lb_1|_\CD\cdot \lb_2|_\CD \cdots \lb_{n}|_\CD)\log p.
$$
Here the intersection is the usual intersection on the projective variety $\CD$ over $\BF_p$.

The definition does not depend on the choice of the rational section $s_{n+1}$.  
It gives a symmetric and multi-linear intersection pairing
$$
\wh\Pic(\CX)^{n+1} \longrightarrow \RR. 
$$
Here $\wh\Pic(\CX)$ denotes the group of isometry classes of hermitian line bundles on $\CX$. 
The intersection number is further compatible with birational morphisms of arithmetic varieties.

\begin{example} \label{degree}
If $\dim \CX=1$, the intersection is just an arithmetic degree map
$$
\wh{\deg}: \wh\Pic(\CX) \longrightarrow \RR. 
$$
If $\CX$ is normal, then it is isomorphic to $\Spec(O_K)$ for some number field $K$.
Then $\CL_1$ is just an $O_K$-module locally free of rank one.
The Hermitian metric $\|\cdot\|$ on $\CL_1(\CC)= \oplus_{\sigma: K\to \CC} \CL_{1,\sigma}(\CC)$
is a collection $(\|\cdot\|_\sigma)_\sigma$ of metrics on the complex line $\CL_\sigma(\CC)$. 
The arithmetic degree is just 
$$
\wh{\deg}(\overline\CL_1)= \log \#(\CL_1/O_Ks) -\sum_{\sigma: K\to \CC} \log \|s\|_\sigma. 
$$
Here $s$ is any non-zero element of $\CL_1$. 
\end{example}

Let $\CX$ be an arithmetic variety and $\CLL=(\CL,\VCV)$ be a  {hermitian line bundle} on $\CX$. We say that $\CLL$ is \emph{nef} if its hermitian metric is semipositive and the intersection number $\CLL\cdot \CCC\geq0$ for any 1-dimensional closed integral subscheme $\CCC$ of $\CX$.

\subsection{Faltings heights}
In this subsection, we review Faltings heights of abelian varieties, which are very crucial in the proof of the Mordell conjecture by Faltings \cite{Fal83}. 

Let $A$ be an abelian variety of dimension $g$ over a number field $K$. Let $\pi\colon\CA\to\Spec O_K$ be the N\'eron model of $A$, and let $e\colon\Spec O_K\to \CA$ be the identity section. 
We refer to \cite{BLR90} for a detailed treatment of N\'eron models.

Define the \emph{Hodge vector bundle} 
$$\underline{\Omega}_{\CA} := e^*\Omega_{\CA/O_K}\cong \pi_*\Omega_{\CA/O_K}$$
 and the \emph{Hodge line bundle} 
$$\underline{\omega}_{\CA} := \det \underline{\Omega}_{\CA} = e^*\omega_{\CA/O_K}\cong \pi_*\omega_{\CA/O_K}.$$
The \emph{Faltings metric} $\Vert\cdot\Vert_{\Fal} $ on $\underline{\omega}_{\CA}$ is defined by 
$$
\Vert\alpha\Vert^2_{\Fal, \sigma}:= \frac{1}{2^g}\left|\int_{A_\sigma(\BC)}\alpha\wedge\overline{\alpha}\right| , 
$$
where $\sigma\colon K\hookrightarrow\BC$ is any embedding,  and
 $\alpha\in\underline{\omega}_\CA\otimes_{\sigma}\BC$
 is viewed as a holomorphic differential from on $A_\sigma(\BC)$ via 
 $\underline{\omega}_\CA\otimes_{\sigma}\BC=\Gamma\left(A_\sigma(\BC), \omega_{A_\sigma(\BC)/\BC}\right)$. 
By construction, \emph{the arithmetic Hodge bundle}
$$ \widehat{\underline{\omega}}_\CA: = (\underline{\omega}_\CA, \VCV_{\Fal}) $$
 is a \hlb\ on $\Spec O_K$.

Define the \emph{Faltings height} of $A$ over $K$ by
 $$ h_{\Fal}^*(A): = \frac{1}{[K:\QQ]}\Adeg\left(\widehat{\underline{\omega}}_\CA\right). $$
Define the \emph{stable Faltings height} of $A$ by
 $$ h_\Fal(A): = h_\Fal(A_{K'}), $$
 where $K'$ is a finite extension of $K$ such that $A$ has semi-abelian reduction over $O_{K'}$. 
 
We refer to \cite[\S7.4]{BLR90} for the notion of semi-abelian reduction. In particular, the field $K'$ in the definition always exists. 
By the normalizing factor, the stable Faltings height $h_\Fal(A)$ is independent of the choice of $K'$.
 

Let $K$ be a number field. Let $C$ be a curve of genus $g>1$ over $K$. 
Denote by $J$ the Jacobian variety of $C$ over $K$.  
The \emph{stable Faltings height} 
$$
h_\Fal(C):= h_\Fal(J)
$$
is an important arithmetic invariant of $C$. 

\subsection{Adelic line bundles}

For Zhang's theory of adelic line bundles on projective varieties over number fields, we refer to \cite{Zha95} and \cite[Appendix A]{YZ}. 

\subsubsection{Local setting}
Let $F$ be a local field. Let $Y$ be a projective variety over $F$, and $M$ be a line bundle on $Y$. An \emph{$F$-metric} of $M$ on $Y$ is a collection of $\overline F$-metrics $\|\cdot\|$ on $M(y)$ over all $y\in Y(\overline F)$, which is continuous and Galois invariant. Namely, for any pair $(U,t)$ consisting of a Zariski open subset $U$ of $Y$ and a section $t$ of $M$ on $U$,  the function $\|t(y)\|$ of $y\in U(\overline F)$ is continuous under the topology induced by that of $F$ and  invariant under the action of $\Aut(\overline F/F)$. 
The pair $\OM=(M,\|\cdot\|)$ is called a \emph{metrized line bundle} on $Y$. 

Assume that $F$ is non-archimedean with a valuation ring $O_F$. Let $(\CY, \CM')$ be an integral model of $(Y,M^{\otimes e})$ over $O_F$ for some positive integer $e$; i.e., $\CY$ is a projective and flat integral scheme over $O_F$, and $\CM'$ is a  line bundle on $\CY$, such that the generic fiber $(\CY_F, \CM'_F)=(Y,M^{\otimes e})$. Then $(\CY, \CM')$ induces an $F$-metric $\|\cdot\|'$ of $M^{\otimes e}$ by the rule that for any $y\in Y(\overline F)$, the closed unit ball of $M^{\otimes e}(y)$ under the metric
is equal to the lattice $\CM'|_{\bar y}$, where $\bar y: \Spec O_{\ol F} \to \CY$ is the unique morphism extending $y:\Spec \overline F \to Y$. 
By taking roots, we obtain a metric $\|\cdot\|=\|\cdot\|'^{1/e}$ of $M$. 
Such an $F$-metric of $M$ is called a \emph{model metric}.
The model metric (or the metrized line bundle) is called \emph{nef} if the  line bundle $\mb'$ is nef on (every irreducible component of) the special fiber of $\CY$. 

Let $\OM=(M,\|\cdot\|)$ be a {metrized line bundle} on $Y$. 
The $F$-metric $\|\cdot\|$ (or the metrized line bundle $\OM$) is called \emph{nef} (or equivalently \emph{semipositive}) if it is a uniform limit of nef model  metrics on $M$. Namely, there exists a sequence $\{\|\cdot\|_{m}\}_m$ of nef model  metrics on $M$ such that $\|\cdot\|_m/\|\cdot\|$ converges uniformly to 1 on $Y(\overline F)$. 
An $F$-metric on a line bundle (or a metrized line bundle) is called \textit{integrable} if it is equal to the quotient of two nef metrics (of different line bundles). 

Let $D$ be a Cartier divisor on $Y$ with support $|D|$. 
A \emph{Green function} of $D$ on $Y$ is a continuous and Galois invariant function $g_{D}:(Y\setminus|D|)(\overline F)\to\RR$ with logarithmic singularity along $D$. Namely, for any pair $(U, f)$ consisting of a Zariski open subset $U$ of $Y$ and a defining equation $f$ of $D$ on $U$,  the function $g_D+\log |f|$ on $(U\setminus |D|) (\overline F)$ extends to a continuous and $\Aut(\overline F/F)$-invariant function on $(U\setminus |D|) (\overline F)$. 

Let $Y^\an$ be the \emph{analytic space} associated to $Y$, which is a Hausdorff, compact, and path-connected topological space. If $F=\CC$, it is the usual complex space $Y(\CC)$; if $F=\RR$, it is the quotient of complex space $Y(\CC)$ by the complex conjugation; if $F$ is non-archimedean, then $Y^\an$
is the \emph{Berkovich analytic space} introduced by \cite{Ber}. 

Denote $d=\dim Y$. Let $(M_1,\|\cdot\|_1), \dots, (M_d,\|\cdot\|_d)$ be integrable metrized line bundles on $Y$. Then we have a \emph{Monge--Amp\`ere measure} 
$c_1(M_1,\|\cdot\|_1)\wedge \cdots \wedge c_1(M_d,\|\cdot\|_d)$ on $Y^\an$. 
If $F=\CC$, it is the classical one defined as the wedge product of the Chern forms;
if $F=\RR$, it is obtained as the push-forward of the classical one on $Y(\CC)$;
if $F$ is non-archimedean, it is the \emph{Chambert-Loir measure} 
 introduced by \cite{CL}.
 In all cases, the total volume of $c_1(M_1,\|\cdot\|_1)\wedge \cdots \wedge c_1(M_d,\|\cdot\|_d)$ is equal to the intersection number $M_1 \cdots M_d$.

\subsubsection{Adelic line bundles}
Let $K$ be a number field, $X$ be a projective variety over $K$, and $L$ be a line bundle on $X$.
An \emph{adelic metric} on $L$ is a \emph{coherent} collection $(\|\cdot\|_v)_{v\in M_K}$ of  $K_v$-metrics $\|\cdot\|_v$ on $L_{K_v}$ over $X_{K_v}$ over all places $v$ of $K$. 
That the collection $(\|\cdot\|_v)_{v\in M_K}$ is \emph{coherent} means that, there exist a finite set $S$ of non-archimedean places of $K$ and a (projective and flat) integral model 
$(\CX,\CL)$ of $(X,L)$ over $\Spec(O_K)\setminus S$, such that the $K_v$-norm $\|\cdot\|_v$ is induced by $(\CX_{O_{K_v}},\CL_{O_{K_v}})$ for all $v\in\Spec(O_K)\setminus S$.

In the above situation, we write $\overline L= (L, (\|\cdot\|_v)_{v\in M_K})$ and call it an \emph{adelic line bundle on} $X$. 
Note that \cite[Appendix A]{YZ} poses stronger continuity conditions on $\|\cdot\|_v$ at non-archimedean places, but these give the same notions of model, nef, and integrable adelic line bundles in the following. 

Let $(\CX, \lbb')$ be a \emph{arithmetic model} of $(X,L^{\otimes e})$ for some positive integer $e$; i.e., $\CX$ is an arithmetic variety over $O_K$, and $\lbb'=(\lb', \|\cdot\|)$ is a hermitian line bundle on $\CX$, such that the generic fiber $(\CX_K, \lb'_K)=(X,L^{\otimes e})$. Then $(\CX, \lbb')$ induces an adelic metric $(\|\cdot\|_v)_v$ on $L^{\otimes e}$, and 
$L$. The metrics at non-archimedean places are clear from the above process, and the metrics at archimedean places are just given by the root of the hermitian metric. 
Such an adelic metric (resp.  adelic line bundle) is called a \emph{model adelic metric} (resp.  \emph{model adelic line bundle}). 
The model adelic metric (resp. model adelic line bundle) is called \emph{nef} if the hermitian line bundle $\lbb'$ is nef on $\CX$. 

An adelic line bundle $\overline L= (L, (\|\cdot\|_v)_v)$ on $X$ is called \emph{nef} if the adelic metric $(\|\cdot\|_v)_v$ is a uniform limit of nef model adelic metrics on $L$. Namely, there exists a sequence $\{(\|\cdot\|_{m,v})_v\}_m$ of nef model adelic metrics on $L$, and a finite set $S$ of non-archimedean places of $K$, such that $\|\cdot\|_{m,v}=\|\cdot\|_{v}$ for any $v\in\Spec(O_K)\setminus S$ and any $m$, and such that $\|\cdot\|_{m,v}/\|\cdot\|_{v}$ converges uniformly to 1 at all places $v$. 
An adelic line bundle is called \textit{integrable} if it is isometric to the tensor quotient of two nef adelic line bundles. 

An \emph{adelic divisor} on $X$ is a pair $\ol D=(D, (g_{D,v})_v)$ consisting of a Cartier divisor $D$ on $X$ and a collection $(g_{D,v})_v$ of Green functions $g_{D,v}$ of $D_{K_v}$ on $X_{K_v}$ for $v\in M_K$, which satisfies a coherence condition. The coherence condition is similar to the case of adelic line bundles, so we omit it here. 

If $\OL$ is an adelic line bundle on $X$, and $s$ is a nonzero rational section of $L$, then we have an adelic divisor
$$
\wh\div(s):=(\div(s), (-\log\|s\|_v)_v).
$$
Conversely, if $\ol D=(D, (g_{D,v})_v)$ is an adelic divisor on $X$, then we have an adelic line bundle $\CO(\OD)$ on $X$ with underlying line bundle $\CO(D)$ and with metric given by $\|1\|_v=e^{-g_{D,v}}$. 
By this correspondence, we have notions of \emph{nef adelic divisors} and \emph{integrable adelic divisors}.

\subsubsection{Intersection numbers}
Let $X$ be a projective variety of dimension $n$ over a number field $K$. 
Denote by $\wh\Pic(X)_{\rm int}$ the group of isometry classes of integrable adelic line bundles on $X$.
By a limit process, the intersection pairing of hermitian line bundles on integral models of $X$ extends to a symmetric and multi-linear intersection pairing
$$
\wh\Pic(X)_{\rm int}^{n+1} \longrightarrow \RR, \quad 
(\overline L_1, \overline L_2, \dots, \overline L_{n+1})
\longmapsto \overline L_1\cdot \overline L_2 \cdots \overline L_{n+1}.
$$
For any closed subvariety $Y$ of dimension $d$ in $X$, and any integrable adelic line bundles $\overline L_1, \dots,  \overline L_{d+1}$ on $X$, 
denote 
$$
\overline L_1\cdot \overline L_2 \cdots \overline L_{d+1} \cdot Y= \overline L_1|_Y\cdot \overline L_2|_Y \cdots \overline L_{d+1}|_Y.
$$
By linear combination, the definition extends to Chow cycles of dimension $d$ on $X$. 

The following example is the adelic version of Example \ref{degree}. 

\begin{example} \label{degree2}
If $X=\Spec(K)$, then the line bundle $L_1$ on $X$ is just a vector space over $K$ of dimension one.
We simply have 
$$
\wh{\deg}(\overline L_1)= -\sum_{v\in M_K} \epsilon_v\log \|s\|_v. 
$$
Here $s$ is any non-zero element of $L_1$, and the degree is independent of the choice of $s$ by the product formula. 
\end{example}

Return to the high-dimensional case, we have an induction formula of Chambert-Loir--Thuillier (cf. \cite[Thm. 4.1]{CLT}).
Let $s_{n+1}$ be a nonzero rational section of $L_{n+1}$ on $X$. 
Then their induction formula gives
$$
 \OL_1\cdots \OL_n\cdot \OL_{n+1} 
 =
 \OL_1\cdots \OL_n\cdot \wdiv(s_{n+1})-\sum_{v\in M_K} \epsilon_v
\int_{X_{K_v}^\an}\big( \log \|s_{n+1}\|_{n+1,v} \big) c_1(\OL_1)_v\cdots c_1(\OL_n)_v. 
$$
Here $c_1(\OL_1)_v\cdots c_1(\OL_n)_v$ denotes the Monge--Amp\`ere measure on $X_{K_v}^\an$ induced by the metrics of $\OL_1,\dots, \OL_{n}$ at $v$.

By the relation between adelic line bundle and adelic divisors, we also have intersection numbers of integrable adelic divisors, or intersection numbers of integrable adelic divisors with integrable adelic line bundles.

Finally, for any adelic line bundle $\OL$ on $X$, we have a \emph{height function} 
$$
h_{\OL}: X(\OK)\lra \RR
$$
defined by 
$$
h_{\OL}(x)=\frac{1}{\deg(x)} \OL\cdot \tilde  x=\frac{1}{\deg(x)} \wh\deg(\OL\cdot \tilde  x), \qquad x\in X(\OK). 
$$
Here $\tilde x$ denotes the closed point of $X$ corresponding to $x$, and $\deg(x)=[K(\tilde x):K]$ denotes the degree of the residue field of $\tilde x$ over $K$. 
It is known that $[K:\QQ]^{-1}h_{\OL}$ is a Weil height function on $X(\OK)$ associated to $L$.

\subsection{Admissible adelic line bundles}
In this subsection, we review Zhang's theory of admissible adelic line bundles on curves over number fields. We refer to \cite{Zha93} and \cite[Appendix A]{Yua21} for more details. 

\subsubsection{Local setting}

Let $F$ be a local field. Let $C$ be a curve of genus $g>1$ over $F$. 
Let $\omega_{C/F}$ be the canonical sheaf. 
Let $\Delta$ be the diagonal of $C^2=C\times_FC$, viewed as a divisor on $C^2$. 

We start with some terminology of metrics on $C^2$.
An $F$-metric $\|\cdot\|_{\Delta}$ of $\CO_{C^2}(\Delta)$ on $C^2$ is called 
\emph{symmetric}
if the Green function 
$$g_\Delta=-\log\|1\|_{\Delta}: (C^2\setminus \Delta)(\overline F)\lra \RR$$ 
is symmetric in the two components of $C^2$.

For any finite extension $F'$ of $F$ and any point $x\in C(F')$, denote
$$
(\CO(x), \|\cdot\|_{x}): = i_{x}^* (\CO(\Delta), \|\cdot\|_{\Delta})
$$
as metrized line bundles on $C_{F'}$.
Here 
$\CO(x)$ is the line bundle on $C_{F'}$ corresponding to $x\in C(F')$, and
$$i_{x}=(x, \mathrm{id}): \Spec {F'}\times_{\Spec F} C \lra C\times_{\Spec F} C$$
is the natural morphism.
It follows that 
$$
g_{x}=-\log\|1\|_{x}:C(\overline F) \setminus \{x\}\lra \RR$$ 
is equal to 
the pull-back of $g_{\Delta}$ via the map
$i_{x}:C(\ol F) \to (C^2)(\ol F)$.
We may also write $g_{x}=g_{\Delta}(x,\cdot)$. 

In the terminology of \cite[Thm. A.1]{Yua21}, there are a unique integrable $F$-metric 
$\|\cdot\|_{a}$ of $\omega_{C/F}$ on $C$ and a unique symmetric integrable $F$-metric 
$\|\cdot\|_{\Delta}$ of $\CO(\Delta)$ on $C^2$ 
satisfying the following properties
 for all finite extensions $F'/F$ and all points $x\in C(F')$:
\begin{enumerate}
\item the equality
$$
(2g-2)c_1(\CO(x), \|\cdot\|_{x})
=c_1(\omega_{C_{F'}/F'}, \|\cdot\|_{a}),
$$
of Monge--Amp\`ere measures holds on $(C_{F'})^\an$;
\item the integral
$$
\int_{(C_{F'})^\an} g_{\Delta}(x,\cdot)\, c_1(\omega_{C_{F'}/F'}, \|\cdot\|_{a})=0;
$$
\item the residue map
$$
(\omega_{C/F} \otimes_{\CO_C} \CO(x) )|_{x}\lra F'
$$
is an isometry, where $F'$ is endowed with the absolute value extending that of $F$.
\end{enumerate}
We call the metric $\VCV_a$ the \emph{admissible metric} of 
$\omega_{C/F}$, which is actually nef. 
We call the Green function $g_\Delta$ the \emph{admissible Green function} of $\Delta$, which is also written as $g_{\Delta,a}$ to emphasize the situation.

Let $D$ be a divisor on $C$. 
The above Green function $g_{\Delta,a}$ induces an \emph{admissible Green function} 
$g_{D,a}:(C\setminus |D|)(\overline F)\to\RR$ of $D$ on $C$. 
For this, it suffices write $D$ as a linear combination of $\ol F$-points of $C$,
and take linear combination of the corresponding Green functions of the form $g_{x}$ above. 
Then we have an integrable $F$-metric $\|\cdot\|_{D,a}$ on $\CO(D)$ by 
the rule  $\|1\|_{D,a}=e^{-g_{D,a}}$.

If $F=\CC$, the metrics are just the original Arakelov metrics on compact Riemann surfaces in \cite{Ara74}. 
If $F=\RR$, they are induced by the Arakelov metrics on $C(\CC)$. 
If $F$ is non-archimedean, the metrics are essentially introduced by Zhang \cite{Zha93} as counterparts of the complex setting.

\subsubsection{Global setting}

Let $K$ be a number field. Let $C$ be a curve of genus $g>1$ over $K$. 
Let $\omega_{C/K}$ be the canonical sheaf. 
Let $\Delta$ be the diagonal of $C^2=C\times_KC$, viewed as a divisor on $C^2$. 
Let $D$ be a divisor on $C$. 

Apply the above constructions to the local field $K_v$ for every place $v$ of $K$.
We obtain an adelic line bundle
$\bar\omega_{C/K,a}=(\omega_{C/K}, (\|\cdot\|_{a,v})_v)$ on $C$, and an adelic line bundle 
$\ol{\CO}(\Delta)_a=(\CO(\Delta), (\|\cdot\|_{\Delta, a,v})_v)$ on $C^2$. 
We also obtain an adelic divisor $\ol\Delta=(\Delta, (g_{\Delta, a,v})_v)$
on $C^2$, an adelic divisor
$\ol D=(D,  (g_{D, a,v})_v)$ on $C$, and an adelic line bundle $\ol\CO(D)_a=\CO(\ol D)$ on $C$. 

All the above objects are integrable. 
For convenience, we say that all of them are \emph{admissible}. 
Moreover, $\bar\omega_{C/K,a}$ is called the \emph{admissible canonical bundle} of $C$ over $K$.

The arithmetic self-intersection number $\bar\omega_{C/K,a}^2\in \RR$ is called the \emph{admissible volume} of $C/K$. It is an important arithmetic invariant of $C$.
It is known that $\bar\omega_{C/K,a}$ is always nef, so the admissible volume $\bar\omega_{C/K,a}^2\geq 0$. 
We will see that the strict positivity $\bar\omega_{C/K,a}^2> 0$ holds, and that the strict positivity is equivalent to the Bogomolov conjecture of $C$.
For convenience, we introduce the \emph{normalized admissible volume}
$$
(\bar\omega_{C/K,a}^2)_\QQ:= \frac{1}{[K:\QQ]} \bar\omega_{C/K,a}^2. 
$$

The following arithmetic Hodge index theorem is a variant of the model case of the original theorem of Faltings \cite{Fal84} and Hriljac \cite{Hri}. 

\begin{theorem}[arithmetic Hodge index theorem]
\label{hodge index}
Let $D$ and $E$ be divisors of degree 0 on $C$. Then
$$
\ol{\CO}(D)_a\cdot \ol{\CO}(E)_a = -2 \pair{D,E}. 
$$
Here the right-hand side is the N\'eron--Tate height pairing on the Jacobian variety of $C$. 
\end{theorem}

The following admissible version of the arithmetic adjunction formula of Arakelov \cite{Ara74} is a major motivation of the definition of the admissible metrics. 

\begin{theorem}[arithmetic adjunction formula]
\label{adjunction}
Let $P$ be rational point of $C$ over $K$. Then
$$
\ol{\CO}(P)_a\cdot \ol{\CO}(P)_a = -\bar\omega_a\cdot \ol{\CO}(P)_a
= -\bar\omega_a\cdot P. 
$$
\end{theorem}
We have the following easy consequences.
\begin{theorem}
\label{canonical height}
Let $P$ and $Q$ be rational points of $C$ over $K$. 
Denote $\bar\alpha=(2g-2)^{-1} \bar\omega_{C/K,a}$, viewed as an adelic $\QQ$-line bundle.  
Then
$$
-2\pair{P,P} = -2g\, \bar \alpha\cdot  P+\bar \alpha^2
$$
and 
$$
-2 \pair{P,Q} = \ol{\CO}(\Delta)_a \cdot (P,Q)-\bar \alpha\cdot  P-\bar \alpha\cdot  Q+\bar\alpha^2. 
$$
\end{theorem}
\begin{proof}
By the arithmetic Hodge index theorem and the arithmetic adjunction formula, 
$$
-2 \pair{P,P}=(\ol P-\bar \alpha)^2
=\ol P^2-2 \bar \alpha\cdot \ol P+\bar \alpha^2
=-\bar\omega\cdot \ol P-2 \bar \alpha\cdot \ol P+\bar \alpha^2
=-2g\, \bar \alpha\cdot  P+\bar \alpha^2,
$$
and
$$
-2 \pair{P,Q}=(\ol P-\bar \alpha)\cdot (\ol Q-\bar \alpha)
=\ol P\cdot \ol Q-\bar \alpha\cdot \ol P-\bar \alpha\cdot \ol Q+\bar\alpha^2.
$$
\end{proof}

\section{Quantitative Vojta inequality}
As mentioned above, Vojta's proof of the Mordell conjecture in \cite{Voj91} actually establishes an inequality concerning relative positions of two algebraic points of $C(\ol K)$ in the Mordell--Weil group $J(\ol K)$. 
In \cite{YYZ}, we have the following quantitative version of Vojta's inequality. 

\begin{theorem}[quantitative Vojta inequality] \label{vojta}
Let $x$ and $y$ be distinct points of $C(\overline K)$ such that 
$$
|x|\geq |y| \geq 1.2\cdot 10^{9}\cdot g^{\frac{7}{3}}
\sqrt{\bar\omega_{C/K,a}^2}
$$
and 
$$
\frac{\langle x,y\rangle}{|x|\cdot |y|} \geq \sqrt{\frac{1.01}{g}}. 
$$
Then 
$$
1.15\, |y|\leq |x| \leq 10^5 g^{\frac52}\, |y|.
$$
\end{theorem}

The major part of the conclusion is $|x| \leq 10^5 g^{\frac52}\, |y|$, which refines Vojta's inequality. 
The extra part $1.15\, |y|\leq |x|$ is a quantitative version of Mumford's inequality in \cite{Mum65} (cf. \cite[\S5.7]{Ser89}).  
Note that R\'emond \cite[Thm. 1.1]{Rem00b} proved a uniform Vojta inequality for subvarieties of abelian varieties, most of whose constants are also explicitly written down.

\subsection{Proof of the quantitative Vojta inequality}

Although it is inspired by the proof of the classical Thue--Siegel--Roth theorem,
Vojta's original proof of the Mordell conjecture in \cite{Voj91} depends heavily on high-dimensional Arakelov geometry of Gillet--Soul\'e \cite{GS90, GS92}.
The most technical part of the proof is an application of of Gillet--Soul\'e's arithmetic Riemann--Roch theorem  to construct a small section, which also involves an estimate of some analytic torsion. Bombieri \cite{Bom90} (based on \cite{Fal91}) wrote a relatively elementary variant of Vojta's proof, which replaced the application of Arakelov geometry by those of Siegel's lemma and classical height theory, at the cost of involving lots of less important extra objects and estimates. We refer to the textbooks 
\cite{HS00, BG,IKM22} for Bombieri's version of the proof.
In the recent work \cite{Yua25}, the author introduced a simplified version of Vojta's original proof, which replaces the role of the arithmetic Riemann--Roch theorem by a quick application of the arithmetic Siu inequality of \cite{Yua08}.

Our proof of Theorem \ref{vojta} follows the idea of \cite{Yua08}. 
By base change, we can assume that $x$ and $y$ are rational points of $C$ over $K$. We will estimate the heights by considering intersection numbers of adelic line bundles on $X=C^2$. 

To illustrate the idea, we start with Mumford's inequality. 
Let $\Delta$ be the diagonal of $X=C^2$, and let $\overline \Delta=(\Delta, (g_v)_v)$  be the admissible adelic divisor on $X$ with underlying divisor $\Delta$.
By Theorem \ref{canonical height}, we have
$$
\ol\Delta\cdot (x,y)=
\frac{1}{g}|x|^2+\frac{1}{g}|y|^2-2 \pair{x, y} + \big(\frac{1}{g}-1\big)\bar\alpha^2.
$$
On the other hand, as the underlying divisor $\Delta$ is effective and does not contain $(x,y)$, we have 
$$
\ol\Delta\cdot (x,y) =\sum_{v\in M_K} g_v(x,y) \geq O(1). 
$$
The error term $O(1)$ is independent of $x,y$, and explicitly estimated in \cite{YYZ}. 
It follows that
$$
\frac{1}{g}|x|^2+\frac{1}{g}|y|^2-2 \pair{x, y} \geq O(1).
$$
This implies Mumford's inequality in the theorem. 

The proof of Vojta's inequality is a significant amplification of the above idea. As above, let $\alpha$ be a line bundle on $C$ with 
$(2g-2)\alpha\cong \omega_{C/K}$, which exists by base change.
Here we write tensor products of line bundles additively. 
Denote by $p_1,p_2:X\to C$ the two projections. Define the \emph{Vojta line bundle} to be
$$
L:= (\delta_1-1) p_1^* \alpha+(\delta_2-1) p_2^* \alpha+ \CO(\Delta)
=\delta_1 p_1^* \alpha+\delta_2p_2^* \alpha+ (\CO(\Delta)-p_1^* \alpha-p_2^* \alpha), 
$$
where $\delta_1,\delta_2$ are positive rational numbers. 
Let $\bar\alpha$ be an adelic line bundle on $C$ with 
$(2g-2)\bar\alpha\cong \bar\omega_{C/K,a}$. 
Define the \emph{Vojta adelic line bundle} to be
$$
\ol L:= (\delta_1-1) p_1^* \bar\alpha+(\delta_2-1) p_2^* \bar\alpha+ \ol\CO(\Delta)_a
= \delta_1 p_1^* \bar\alpha+\delta_2 p_2^* \bar\alpha+ (\ol\CO(\Delta)_a-p_1^* \bar\alpha-p_2^* \bar\alpha).  
$$
Note that the case $\delta_1=\delta_2=1$ recovers the above divisor $\ol\Delta$ used to prove Mumford's inequality, but we will take quite different $(\delta_1, \delta_2)$ in the following.  
By Theorem \ref{canonical height}, the height 
$$
\ol L\cdot (x,y) =\frac{1}{g}\delta_1  |x|^2
+\frac{1}{g} \delta_2 |y|^2 
-2 \pair{x,y}
+\big(\frac{\delta_1+\delta_2}{2g}-1\big) \bar\alpha^2.
$$
The key is to have a suitable lower bound of the left-hand side. 

Assume that $\delta_1\delta_2>g$ in the following. This gives 
$$
L^2=2\delta_1\delta_2-2g>0.  
$$
This implies that the line bundle $L$ is big on $X$. 
Write $\ol\CO(\Delta)_a$ as a difference of two nef adelic line bundles on $X$, and apply Yuan's arithmetic Siu inequality and an adelic version of the Minkowski theorem. We find a positive integer $m$ divisible by the determinators of $\delta_1$ and $\delta_2$ such that $m\ol L$ has a nonzero global section $s$ of ``small supremum norms''.
If the section $s$ does not vanish at $(x,y)$, using it to compute intersection numbers gives 
$$\ol L \cdot (x,y) \geq \text{error term}.$$
If the section $s$ vanishes at $(x,y)$, we follow Vojta's idea to bound the index (or multiplicity) of $s$ at $(x,y)$ by Dyson's lemma, and some extra argument still gives 
$$\ol L \cdot (x,y) \geq \text{error term}.$$
The error term is explicit in \cite{YYZ}. 

Combining the equality and the inequality for $\ol L \cdot (x,y)$, we have
$$\frac{1}{g}\delta_1  |x|^2
+\frac{1}{g} \delta_2 |y|^2 
-2 \pair{x,y} \geq \text{error term}.$$
Then the quantitative Vojta inequality is obtained by taking 
$$\delta_1\approx 1.2\sqrt{g}\frac{|y|}{|x|},\quad
\delta_2\approx 1.2\sqrt{g}\frac{|x|}{|y|}.$$

\subsection{Counting large points}\label{sec large}

With Theorem \ref{vojta}, we can bound the cardinality of
$$C(K)_{\rm large}
:=\{x\in C(K):|x|\geq 1.2\cdot 10^{9}\cdot g^{\frac{7}{3}}
\sqrt{\bar\omega_{C/K,a}^2}\}$$ 
as follows. 
Note that $C(K)$ is viewed as a subset of $J(K)$ by the map $i_\alpha:x\to x-\alpha$.
The real vector space 
$V=J(K)_\RR$ has a finite dimension $r=\rank\, J(K)$ by the Mordell--Weil theorem. 
Endow $V$ with the norm $|x|=\sqrt{\hat h(x)_K}$, and thus it is isometric to the Euclidean space $\RR^r$.  
For any two vectors $u,v\in V$,
denote by $\angle(u, v)$ the angle between them, which is set to be 0 if $u=0$ or $v=0$. 
Then the second condition of Theorem \ref{vojta} is equivalent to 
$$
\angle(x, y) \leq \arccos \sqrt{\frac{1.01}{g}}. 
$$

For any nonzero vector $u\in V$, denote the cone
$$
\mathrm{cone}(u):=\{y\in V: \angle(u,y)\leq \frac12 \arccos \sqrt{\frac{1.01}{g}}\}.
$$
Then $V$ is covered by finitely many cones 
$\mathrm{cone}(u_1), \dots, \mathrm{cone}(u_N)$, as a consequence of the compactness of the unit sphere of $V$.
For every $i$, 
order the points of $\Sigma_i\cap C(K)_{\rm large}$ as 
$$
|x_1| \leq |x_2| \leq |x_3| \leq \cdots. 
$$
By Theorem \ref{vojta}, we have $|x_1|\leq |x_i|\leq 10^5 g^{\frac52} |x_1|$
and $|x_{i+1}|\geq 1.15 |x_i|$. As a consequence, we have 
$$
|\Sigma_i\cap C(K)_{\rm large}| \leq \log_{1.15}(10^5 g^{\frac52})+1.
$$
This gives 
$$
|C(K)_{\rm large}| \leq \left(\log_{1.15}(10^5 g^{\frac52})+1\right)N. 
$$
By some elementary argument, the smallest possible $N$ grows approximately as $1.15^{r}$.  
By more delicate techniques of sphere packing, which especially applies results of Rankin \cite[Thm. 2]{Ran55} and 
Kabatjanski\u i--Leven\v ste\u in \cite{KL}, 
we can improve the bound to 
$$
|C(K)_{\rm large}| \leq  2.4\cdot 10^{4} g^8 \big(1+\frac{3\log g}{g}\big)^{r}.
$$

\section{The Bogomolov conjecture}

Let $C$ be a curve of genus $g>1$ over a number field $K$. 
The Bogomolov conjecture, proposed by Bogomolov  and
proved by Ullmo, is the following statement. 

\begin{theorem}[Ullmo \cite{Ull}]
\label{ullmo}
For any divisor $\beta$ of degree 1 on $C_{\ol K}$, there is a constant $\epsilon>0$ such that  
$$\#\{x\in C(\ol K): \hat h(x-\beta)\leq \epsilon\}<\infty$$
\end{theorem}

If $\epsilon=0$, then the theorem is a case of the Manin--Mumford conjecture (proved by \cite{Ray83a,Ray83b}). In fact, Bogomolov \cite{Bog81}  proposed his conjecture to strengthen the Manin--Mumford conjecture. 
The proof of the conjecture by Ullmo \cite{Ull} is based on the equidistribution theorem of Szpiro--Ullmo--Zhang \cite{SUZ}. 
Inspired by Ullmo's idea, Zhang \cite{Zha98} proved an extension of the Bogomolov conjecture to subvarieties of abelian varieties. 

There is a less well-known proof of Theorem \ref{ullmo} through the positivity of the admissible volume $\bar\omega_{C/K,a}^2$ by  Zhang \cite{Zha93, Zha10}, Faber \cite{Fab09}, Cinkir \cite{Cin11}, and de Jong \cite{dJ18}. 
We will describe this approach later.

There is an analogue of the conjecture, called the geometric Bogomolov conjecture,  for subvarieties of abelian varieties over finitely generated fields. 
The proof of Ullmo and Zhang does not work in the geometric case directly due to the lack of archimedean places, while the above approach of 
\cite{Zha93, Zha10, Fab09, Cin11, dJ18} proves the case of curves in Jacobian varieties in the geometric case.  
However, by careful analysis of the equilibrium measures at non-archimedean places,  Gubler \cite{Gub07} proved the geometric Bogomolov conjecture for 
 abelian varieties with totally degenerate reduction at some place. 
Following the strategy of \cite{Gub07}, 
Yamaki \cite{Yam16, Yam17, Yam18} reduced the geometric Bogomolov conjecture to the case of abelian varieties with good reduction everywhere, and proved the conjecture for subvarieties of abelian varieties with dimension 1 or codimension 1.
By a completely different method using Betti maps of complex abelian schemes, Gao--Habegger \cite{GH19} and Cantat--Gao--Habegger--Xie \cite{CGHX} proved the geometric Bogomolov conjecture for characteristic 0. 
Finally, Xie--Yuan \cite{XY22} proved the geometric Bogomolov conjecture in all cases using Yamaki's reduction theorem.

\subsection{The essential minimum}
In this subsection, we introduce the proof of Theorem \ref{ullmo} following the works of Zhang \cite{Zha93, Zha10}, Cinkir \cite{Cin11}, and de Jong \cite{dJ18}. 
As before, let $\alpha$ be a line bundle on $C$ with $(2g-2)\alpha=\omega_{C/K}$, which exists by a base change, and let $\bar\alpha$ be an adelic line bundle on $C$ with $(2g-2)\bar\alpha=\bar\omega_{C/K,a}$. 
By base change, we can further assume that the divisor $\beta$ is defined over $K$. 
Denote by $\bar\beta$ the admissible extension of $\beta$.
We first have the following nice result. 

\begin{lemma}
\label{representative}
Define an adelic line bundle on $C$ by
$$
\OL_\beta=(g-1)\bar\alpha+\bar\beta-\CO_C(\frac12 \bar\beta^2),
$$
where $\CO_C(t)$ is the pull-back of an adelic line bundle on $\Spec K$ of degree $t$ to $C$. 
Then $\OL_\beta$ is nef with 
$$
h_{\OL_\beta}(x)=  \hat h(x-\beta)_K, \quad \forall x\in C(\OK). 
$$
\end{lemma}
\begin{proof}
For the height identity, by base change, it suffices to assume that $x\in C(K)$.
By Theorem \ref{canonical height}, 
$$ 
\hat h(x-\beta)_K=-\frac12 (\bar x-\bar\beta)^2
=-\frac12\bar x^2+\bar\beta\cdot \bar x-\frac12\bar\beta^2
=(g-1)\bar\alpha\cdot x+\bar\beta\cdot x-\frac12\bar\beta^2
=\OL_\beta\cdot x.
$$
For the nefness, it suffices to note that the adelic metric of $\OL_\beta$ is semipositive at every place since it is admissible, and that $h_{\OL_\beta}(x)=  \hat h(x-\beta)_K\geq0$ for every $x\in C(\OK)$. 
\end{proof}

For any adelic line bundle $\OL$ on $C$, following Zhang \cite{Zha95}, we introduce the \emph{essential minimum} 
$$
e_1(C, \OL)= \liminf_{x\in C(\OK)} h_{\OL}(x)$$
and the \emph{absolute minimum}
$$ 
e_2(C, \OL)= \inf_{x\in C(\OK)} h_{\OL}(x).
$$
The theorem of successive minimum of Zhang \cite{Zha95} asserts that, if $\OL$ is nef and $\deg(L)>0$, then 
$$
e_1(C, \OL) \geq \frac{1}{2 \deg(L)} \OL^2 \geq \frac12(e_1(C, \OL)+e_2(C, \OL))
\geq \frac12e_1(C, \OL).
$$

By Lemma \ref{representative}, we see that Theorem \ref{ullmo} is equivalent to $e_1(C,\OL_\beta)>0$. 
By the theorem of successive minima, this is further equivalent to 
$\OL_\beta^2>0. $
On the other hand, we have 
\begin{align*}
\OL_\beta^2=&\ \left((g-1)\bar\alpha+\bar\beta-\CO_C(\frac12 \bar\beta^2)\right)^2\\
=&\ (g-1)^2\bar\alpha^2+2(g-1)\bar\alpha\cdot\bar\beta+\bar\beta^2-g \bar\beta^2
= g(g-1)\bar\alpha^2-(g-1)(\bar\alpha- \bar\beta)^2\ \geq g(g-1)\bar\alpha^2.
\end{align*}
Here the inequality follows from the arithmetic Hodge index theorem 
in Theorem \ref{hodge index},
and the equality holds if and only if $\beta=\alpha+t$ for a torsion element 
$t\in \Pic(C)$. 
Therefore, we have proved that Theorem \ref{ullmo} is equivalent to
$\bar\omega_{C/K,a}^2>0$.  
This recovers Zhang \cite[Cor. 5.7]{Zha93}.

\subsection{Positivity of the admissible volume}

The positivity of the admissible volume $\bar\omega_{C/K,a}^2$ is built upon Zhang's $\varphi$-invariant introduced in \cite{Zha10}. Recall that the global $\varphi$-invariant
$$
\varphi(C)= \sum_{v\in M_K} \epsilon_{v} \varphi_v(C),
$$
where the 
the local $\varphi$-invariant  is defined as the integral
$$
\varphi_v(C)
=-\int_{(C_{K_v}\times C_{K_v})^\an}g_{\Delta,a} \, c_1(\CO(\ol \Delta)_a)^2.
$$
The summation has only finitely many nonzero terms, since $\varphi_v(C)=0$ if $v$ is non-archimedean and $C$ has potentially good reduction at $v$. 
If $v$ is archimedean, \cite[Prop 2.5.3]{Zha10} proves that $\varphi_v(C)>0$ by completing it as a sum of squares in terms of the heat kernel. 
If $v$ is non-archimedean and $C$ does not have potentially good reduction at $v$, then $\varphi_v(C)$ is an invariant of the reduction graph of $C$ at $v$, and 
 Cinkir \cite[Thm. 2.11]{Cin11} proves  $\varphi_v(C)> 0$ in this case.
As a consequence, we have $\varphi(C) >0$. 

Finally, the following theorem finishes the proof of the Bogomolov conjecture in Theorem \ref{ullmo}. 
\begin{theorem}[de Jong \cite{dJ18}, Wilms \cite{Wil}]
\label{wilms}
Let $C$ be a  curve of genus $g>1$ over a number field $K$. Then
$$
\bar\omega_{C/K,a}^2 \geq \frac{\max\{g-1,2\}}{2g+1}  \varphi(C) >0. 
$$
\end{theorem}

A version of the theorem with a weaker coefficient was proved by \cite[Cor. 1.4]{dJ18}. 
The current version is proved by \cite[Thm. 1.2]{Wil} by applying the Hodge index theorem of adelic line bundles of Yuan--Zhang \cite{YZ17}.

\section{The quantitative Bogomolov conjecture}

Let $C$ be a curve of genus $g>1$ over a number field $K$. As before, denote by $\alpha$ a line bundle on $C_{\OK}$ with $(2g-2)\alpha\cong \omega_{C/K}$. 
In this section, we introduce various versions of the uniform Bogomolov conjecture, which especially includes the quantitative version of Yu--Yuan--Zhou \cite{YYZ}. 

\subsection{The uniform Bogomolov conjecture}

In the Bogomolov conjecture in Theorem \ref{ullmo}, by varying the curve $C$, 
we obtain the following uniform Bogomolov conjecture. 

\begin{theorem}[Dimitrov--Gao--Habegger \cite{DGH21}, K\"uhne \cite{Kuh}, Yuan \cite{Yua21}]
\label{uniform bog}
Let $g>1$ be an integer. 
Then there are positive constants $c_3(g)$ and $c_4(g)$ depending only on $g$ such that  for any curve $C$ of genus $g$ over a number field $K$, and any divisor $\beta$ of degree 1 on $C_{\ol K}$, we have   
$$\#\{x\in C(\ol K): \hat h(x-\beta)\leq c_3(g) (\max\{h_\Fal(C),1\}+\hat h(\beta-\alpha)) \}
\leq c_4(g).$$
\end{theorem}

In the case $\beta$ is a point of $C_{\ol K}$, the theorem is due to Dimitrov--Gao--Habegger \cite{DGH21} and K\"uhne \cite{Kuh} without the ``extra term'' $\hat h(\beta-\alpha)$. The above form of the theorem is proved by Yuan \cite{Yua21}. 

The proof of \cite{DGH21,Kuh} and that of \cite{Yua21} are both through arguments on universal families $X\to S$ of curves over number fields.  
The work \cite{DGH21} proves a slightly weaker form of Theorem \ref{uniform bog}, where their first step is a non-degeneracy result proved via functional transcendence and o-minimality, 
and their second step is a height inequality which converts the geometric result (non-degeneracy) to an arithmetic result (bounding algebraic points of small heights). 
Then \cite{Kuh} proved an equidistribution theorem over quasi-projective variety and applied it to strengthen the result of \cite{DGH21}. 
The equidistribution theorem generalizes the original theorem of Szpiro--Ullmo--Zhang \cite{SUZ}, and is proved by Yuan--Zhang \cite{YZ} independently. 
In some sense, the proof of \cite{DGH21,Kuh} is a family version of Ullmo's proof of Theorem \ref{ullmo}.

On the other hand, the proof of \cite{Yua21} is a family version of the proof of Theorem \ref{ullmo} by Zhang \cite{Zha93, Zha10}, Cinkir \cite{Cin11}, and de Jong \cite{dJ18}. It is based on the theory of adelic line bundles on quasi-projective varieties developed by Yuan--Zhang \cite{YZ}, and works over number fields and function fields simultaneously. 
The first step of \cite{Yua21} is to construct the admissible canonical bundle $\bar\omega_{X/S,a}$ for a quasi-projective family $X\to S$ of curves over a number field, which extends Zhang's admissible canonical bundles of projective curves. 
The second step is to prove that $\bar\omega_{X/S,a}$ is big if the family has maximal variation in the moduli space.
The third step uses the bigness of $\bar\omega_{X/S,a}$ to bound algebraic points of small heights.

\subsection{The quantitative Bogomolov conjecture}

Both of the above proofs of Theorem \ref{uniform bog} are based on moduli spaces of curves and seem difficult to give explicit constants of the theorem. 
However, in the case of function fields, Looper--Silverman--Wilms  proved the uniform geometric Bogomolov conjecture independently, and their result has explicit constants as follows.  

\begin{theorem}[Looper--Silverman--Wilms \cite{LSW}] \label{quant bog function field}
Let $K$ be a finitely generated field of transcendence degree 1 over a field $k$.
Let $C$ be a non-isotrivial curve of genus $g>1$ over $K$. 
For any $\ds 0\leq \epsilon < \frac{1}{4g^2-1}$, we have
$$\#\{x\in C(\ol K): \hat h(x-\beta)\leq \epsilon\, \bar\omega_{C/K,a}^2\}
\leq  \frac{16g^4+36g^2-26g-2}{(1-(4g^2-1)\epsilon)(g-1)^2} +1.$$
\end{theorem}

The proof of \cite{LSW} estimates heights by Zhang's admissible adelic line bundles on the fixed curve $C$ (instead of moduli spaces), and eventually reduces the problem to some inequalities between analytic invariants of $C$ at all places of $K$. Due to the difficulty to prove these inequalities at archimedean places (or equivalently on compact Riemann surfaces), the result of \cite{LSW} does not cover the case of number fields. 
Yu--Yuan--Zhou \cite{YYZ} has managed to prove all these desired inequalities at archimedean places and carry the proof to number fields. 
This gives the following arithmetic version of Theorem \ref{quant bog function field} and 
 quantitative version of Theorem \ref{uniform bog}. 

\begin{theorem}[Yu--Yuan--Zhou \cite{YYZ}]
\label{quant bog}
Let  $C$  be a curve of genus $g>1$ over a number field $K$, and $\beta$ be a divisor of degree 1 on $C_{\ol K}$. Then
$$\#\left\{x\in C(\ol K): |x-\beta|\leq \frac{1}{\sqrt{16g}} \sqrt{\bar\omega_{C/K,a}^2}\right\}
\leq 6.5\cdot 10^{11}\cdot g^{\frac{17}{3}}$$
and   
$$\#\left\{x\in C(\overline K):|\beta-\alpha|\le|x-\alpha|\le2 |\beta-\alpha|,\, \angle(x-\alpha,\beta-\alpha)\leq \arccos \sqrt{\frac{1.13}{g}}\right\}
    \leq 
    3.2\cdot 10^{11}g^{\frac{17}{3}}.$$
\end{theorem}

The first inequality is an analogue of Theorem \ref{quant bog function field}, and its has a much larger coefficient $6.5\cdot 10^{11}$ due to the larger coefficients from archimedean places. 
The second inequality is a vast refinement of the extra term of  Theorem \ref{uniform bog} from \cite{Yua21}. 
In fact, the original extra term gives a bound on the number of $x\in C(\ol K)$ such that $|x-\beta|$ is smaller than some multiple of $|\beta-\alpha|$, and the current result refines it by adding extra conditions on angles.

Note that Theorem \ref{uniform bog} uses the modified Faltings height $\max\{h_\Fal(C),1\}$, while Theorem \ref{quant bog} uses the normalized admissible volume $(\bar\omega_{C/K,a}^2)_\QQ$. 
These two arithmetic invariants of $C$ are actually equivalent for many purposes. In fact, we have the following theorem from \cite[Thm. 1.4]{Yua21}. 

\begin{theorem}[Yuan \cite{Yua21}] \label{fiberwise11}
Let $g>1$ be an integer. 
Then there are positive constants $c_5(g)$ and $c_6(g)$ depending only on $g$ such that for any curve $C$ of genus $g$ over a number field $K$, we have
$$
c_5(g) \cdot \max\{h_\Fal(C),1\} \leq  (\bar\omega_{C/K,a}^2)_\QQ \leq  c_6(g) \cdot \max\{h_\Fal(C),1\}.
$$ 
\end{theorem}

The result is also valid over function fields, and we refer to \cite[Thm. 1.4]{Yua21} for more details.

\subsection{Proof of the quantitative Bogomolov conjecture}

Let us describe the idea to prove Theorem \ref{quant bog}. 
As mentioned above, the global part of the argument is duplicated from the proof of 
Theorem \ref{quant bog function field} of \cite{LSW}, which is based on a quadratic identity. 

As in the theorem, let $\beta$ be a divisor of degree 1 on $C_\OK$. 
Assume that $x_1,\dots, x_n$ are $n$ distinct points of $C(\OK)$ with small heights 
$\hat h(x_i-\beta)$. We hope to give an upper bound of $n$. By base change, we can assume that $x_1,\dots, x_n$ and $\beta$ are defined over $K$. 
Extend $x_i$ to the admissible adelic divisor $\bar x_i$ on $C$. 
Take $\bar\alpha=(2g-2)^{-1}\bar\omega_{C/K,a}$ as before. 
By the arithmetic Hodge index theorem (cf. Theorem \ref{hodge index}) and the arithmetic adjunction formula (cf. Theorem \ref{adjunction}), we have
$$
\hat h(x_i-x_j)_K
=-\frac12 (\bar  x_i-\bar x_j)^2
=-\frac12(\bar x_i^2-2\bar x_i\cdot \bar x_j+\bar x_j^2)
=(g-1) \bar\alpha\cdot x_i+ (g-1) \bar\alpha\cdot x_j+\bar x_i\cdot \bar x_j,
$$
and thus
$$
\sum_{1\leq i<j\leq n} \hat h(x_i-x_j)_K
=(g-1)(n-1)\sum_{i}  \bar\alpha\cdot x_i +\sum_{i<j} \bar x_i\cdot \bar x_j.
$$
Similarly, we have
$$
\hat h(x_1+\cdots +x_n-n\alpha)_K
=-\frac12 (\bar x_1+\cdots +\bar x_n-n\bar\alpha)^2
=(n+g-1)\sum_{i}  \bar\alpha\cdot x_i- \sum_{i<j} \bar x_i\cdot \bar x_j -\frac{1}{2} n^2 \bar\alpha^2.
$$
Canceling the term $\ds \sum_{i}  \bar\alpha\cdot x_i$ in the last two expressions, we get the quadratic identity
$$
\sum_{1\leq i<j\leq n} \hat h(x_i-x_j)_K
=\frac{(g-1)(n-1)}{n+g-1}\hat h(x_1+\cdots +x_n-n\alpha)_K
+\frac{gn}{n+g-1}\sum_{i<j} \bar x_i\cdot \bar x_j 
+ \frac{(g-1)(n-1)n^2}{2(n+g-1)} \bar\alpha^2.
$$

Assume that $|x_i-\beta|^2=\hat h(x_i-\beta)_K<t$ for a small positive number $t$, and we want to have an upper bound of $n$ in terms of $t$. 
We first have
$$
\hat h(x_i-x_j)_K = |x_i-x_j|^2 \leq (|x_i-\beta|+|x_j-\beta|)^2 
\leq (\sqrt t+\sqrt t)^2=4t. 
$$

To illustrate the idea, fix $C$ and $g$, and let $n$ increase for the moment. We will prove a key bound of the form
$$
\sum_{i<j} \bar x_i\cdot \bar x_j \geq -O((n\log n)\bar\alpha^2).
$$
It follows that the inequality becomes 
$$
2n^2 t
\geq  (g-1)\hat h(x_1+\cdots +x_n-n\alpha)_K+\frac{g-1}{2} n^2\bar\alpha^2-O((n\log n)\bar\alpha^2).
$$
To get the first inequality of Theorem \ref{quant bog}, note the positivity $\hat h(x_1+\cdots +x_n-n\alpha)_K\geq0$, so the inequality gives an upper bound of $n$ 
as long as $\ds 2t<\frac{g-1}{2} \bar\alpha^2$. 
For the second inequality of Theorem \ref{quant bog}, 
note that for small $t$, 
$$
|x_1+\cdots +x_n-n\alpha| \geq 
n|\alpha-\beta|-|x_1-\beta|-\cdots -|x_n-\beta|
\geq n (|\alpha-\beta|-\sqrt t).
$$
Then for $t< |\alpha-\beta|^2$, we have 
$$
(g-1)\hat h(x_1+\cdots +x_n-n\alpha)_K \geq (g-1) n^2 (|\alpha-\beta|-\sqrt t)^2.
$$
Assume that 
$$
(g-1) n^2 (|\alpha-\beta|-\sqrt t)^2 \geq 2n^2 t
$$
which is satisfied by a relation of the form $t<c(g) |\alpha-\beta|^2$. Then the inequality implies 
$$
0 \geq  \frac{g-1}{2} n^2\bar\alpha^2-O((n\log n)\bar\alpha^2).
$$
This gives an upper bound of $n$ as in the second inequality of Theorem \ref{quant bog}.

It remains to prove an explicit bound of type
$$
\sum_{i<j} \bar x_i\cdot \bar x_j \geq -O((n\log n)\bar\alpha^2).
$$
Note that 
$$
\sum_{i<j} \bar x_i\cdot \bar x_j =\sum_{v\in M_K}\sum_{i<j} \epsilon_{v} g_{v}(x_i, x_j),
$$
where $g_v:(C^2\setminus \Delta)(\ol K_v)\to \RR$ is the admissible Green function of $\Delta$ at $v$.
By Theorem \ref{wilms}, we have
$$
\bar\omega_{C/K,a}^2 \geq \frac{\max\{g-1,2\}}{2g+1} \sum_{v\in M_K} \varphi_v(C). 
$$
The problem is reduced to find a suitable positive constant $c(n,g)$ independent of $C$ and $v$ such that 
$$
\mathrm{FE}_{v}(n,C)+ c(n,g)\, \varphi_v(C) \geq 0.
$$
Here the Faltings--Elkies invariant
$$
\mathrm{FE}_{v}(n,C)=\inf_{x_1,\dots, x_n\in C(K_v)} \sum_{1\leq i<j\leq n} g_{v}(x_i, x_j).
$$
Note that the problem is purely local for $v$. 

If $v$ is a non-archimedean place at which $C$ has potential good reduction, then we simply have $g_{v}(x_i, x_j)\geq0$ and $\varphi_v(C)=0$. 
If $v$ is an arbitrary non-archimedean place, both invariants can be computed in terms of the reduction dual graph. 
In this case, a lower bound of $\varphi_v(C)$ is given by Cinkir \cite{Cin11}, and a suitable lower bound of $\mathrm{FE}_{v}(n,C)$ is given by \cite{LSW}. 
Their combination gives $c(n,g)$ at $v$. 

If $v$ is archimedean, the term $\mathrm{FE}_{v}(n,C)$ is an invariant of the compact Riemann surface $C_v(\CC)$. Lower bounds of $\mathrm{FE}_{v}(n,C)$ was previously obtained by Faltings \cite{Fal84} and Elkies (cf. \cite[\S VI.5]{Lan88}), but their bounds are not uniform as the Riemann surface varies. 

In the following, let $X$ be a connected compact Riemann surface of genus $g>1$. 
Then the invariants $\mathrm{FE}(n,X)$ and $\varphi(X)$ are defined for $X$.
In \cite{YYZ}, these two invariants are bounded below in terms of the \emph{systole} $\mathrm{sys}(X)$ of $X$, i.e. the length of the shortest closed geodesic of $X$ under the hyperbolic metric of constant scalar curvature $-1$.
Note that $\mathrm{sys}(X)>0$ is a canonical invariant of $X$, which goes to 0 as $X$ degenerates to the boundary of the moduli space of Riemann surfaces of genus $g$. 
In \cite{YYZ}, we have proved 
$$
\mathrm{FE}(n,X) \geq -\frac{1}{4} n\log n - 400 \pi^3(g-1)^3  n\cdot 
\max\left\{1, \mathrm{sys}(X)^{-1}\right\}
$$
and 
$$
\varphi(X)\geq \max\left\{ \frac{1}{15625} g^{-\frac13},  \frac{1}{343750} g^{-\frac53} \mathrm{sys}(X)^{-1}\right\}.
$$
Recall that \cite[Thm. 3.10]{Yua21} proves that $\varphi(X)$ goes to infinity as $X$ degenerates to the boundary of the moduli space, and confirms that $\varphi(X)>c_7(g)$ for some positive constant $c_7(g)$ depending only on $g$ by a compactness argument. 
The second inequality gives an explicit value $\ds c_7(g)=\frac{1}{15625} g^{-\frac13}$, and measures the growth of $\varphi(X)$ to infinity in terms of $\mathrm{sys}(X)^{-1}$. 
A combination of these two inequalities gives an inequality 
$$
\mathrm{FE}(n,X) + (7\cdot 10^4 g^{\frac13} n\log n + 5\cdot 10^9 g^{\frac{14}{3}} n) \varphi(X) \geq 0.
$$
The constant are further improved with some more delicate estimates.
This explains the proof of Theorem \ref{quant bog}.

\subsection{Medium points and small points}
Let us return to the proof of Theorem \ref{quant Mordell}. 
In \S\ref{sec large}, with the quantitative Vojta inequality, we have already obtained an explicit upper bound of  
the order of
$$C(K)_{\rm large}
=\left\{x\in C(K):|x|\geq 1.2\cdot 10^{9}\cdot g^{\frac{7}{3}} \sqrt{\bar\omega_{C/K,a}^2}\right\}.$$ 
It remains to estimate the order of the complement
$C(K)\setminus C(K)_{\rm large}.$

A rough bound of this set is obtained as follows. 
Recall that $V=J(K)_\RR$ is endowed with the norm $|x|=\sqrt{\hat h(x)_K}$. 
Denote by $B(x,R)$ the closed ball in $V$ of center $x\in V$ and radius $R$. 
For simplicity, define
$$
R_1=\frac{1}{\sqrt{16g}} \sqrt{\bar\omega_{C/K,a}^2}, \quad \
R_2=1.2\cdot 10^{9}\cdot g^{\frac{7}{3}}
\sqrt{\bar\omega_{C/K,a}^2}.$$ 
For every $x\in C(K)\setminus C(K)_{\rm large}$, the ball $B(x, R_1)$ is contained in the ball $B(0, R_1+R_2)$. By Theorem \ref{quant bog}, 
 every point of $B(0, R_1+R_2)$ is covered by at most $6.5\cdot 10^{11}\cdot g^{\frac{17}{3}}$ such balls. 
By comparing volumes, we have 
$$
6.5\cdot 10^{11}\cdot g^{\frac{17}{3}}\cdot  \vol(B(0, R_1+R_2))
\geq |C(K)\setminus C(K)_{\rm large}| \cdot \vol(B(0, R_1)). 
$$
This gives 
$$
 |C(K)\setminus C(K)_{\rm large}|  \leq 6.5\cdot 10^{11}\cdot g^{\frac{17}{3}} \cdot (1+R_2/R_1)^r
 =6.5\cdot 10^{11}\cdot g^{\frac{17}{3}} \cdot \left(1+4.8\cdot 10^9 g^{\frac{17}{6}}\right)^r.
$$
This bound is too big, comparing the expectation that the Vojta constant tends to 1 as $g$ tends to infinity. 

To lower the constant, we divide the set $C(K)\setminus C(K)_{\rm large}$ into the  subsets
$$C(K)_{\rm small}
=\{x\in C(K):  |x|\leq R_1\}$$
and 
$$C(K)_{\rm medium}
=\{x\in C(K): R_1 <  |x| < R_2\}.$$
By Theorem \ref{quant bog}, 
we have a bound 
$$
|C(K)_{\rm small}| \leq  6.5\cdot 10^{11}\cdot g^{\frac{17}{3}}. 
$$
To estimate 
$C(K)_{\rm medium}$, 
we use the second inequality of Theorem \ref{quant bog}. 
We cover the set $C(K)_{\rm medium}$ by the union of the sets
$$C(K)_{2^i R_1, 2^{i+1} R_1}
=\{x\in C(K): 2^i R_1  \leq  |x| < 2^{i+1} R_1\}$$
for $i=1,\dots, \lceil\log_2(R_2/R_1)\rceil$. Then we bound the order of each $C(K)_{2^i R_1, 2^{i+1} R_1}$ separately, where we also need spherical packing to control angles. 
This eventually gives  
$$|(C\cap\Lambda)_{\mathrm{medium}}|
\leq 1.64\cdot10^{13}g^7\left(1+\frac{5}{2\sqrt2g}\right)^{n}.$$

\section*{Acknowledgments.}
The works of the author in this survey are greatly influenced by the previous works of Vojta, Zhang, de Jong, Dimitrov--Gao--Habegger, K\"uhne, Looper--Silverman--Wilms, and Yuan--Zhang. The author is also grateful to his collaborators Jiawei Yu, Shou-Wu Zhang, and Shengxuan Zhou. The author thanks Chunhui Liu and Jiawei Yu for reading an early draft of this survey.

The author would like to thank the support of the China--Russia Mathematics Center. The author is supported by grants NO. 12250004 and NO. 12321001
from the National Science Foundation of China, and by the Xplorer Prize from the New Cornerstone Science Foundation.


\end{document}